\documentclass[authoryear,final,3p,times]{elsarticle}
\usepackage[utf8]{inputenc}
\usepackage{amsfonts}
\usepackage{amsthm}
\usepackage{amsmath}
\usepackage{xcolor}
\usepackage[english]{babel}

\usepackage{geometry}
\usepackage{color}
\usepackage{soul}
\usepackage{multirow}
\usepackage{natbib}
\usepackage{pdfpages}
\usepackage[english]{babel} 

\usepackage{graphicx}
\usepackage{amsmath}
\usepackage{cancel}

\usepackage{stmaryrd}
\usepackage{hyperref} 
\usepackage{xr}
\usepackage{xcolor}

\usepackage[markup=underlined]{changes}
\definechangesauthor[color=red]{e}

\newcommand{\bS}{\text{\emph{\textbf{\textrm{S}}}}}

\newcommand{\bI}{\text{\emph{\textbf{\textrm{I}}}}}
\newcommand{\bD}{\text{\emph{\textbf{\textrm{D}}}}}
\newcommand{\diffrate}{\delta}

\newcommand{\un}{\text{\emph{\textbf{1}}}}
\newcommand{\zero}{\text{\emph{\textbf{0}}}}

\newcommand{\bX}{\text{\emph{\textbf{X}}}}
\newcommand{\bx}{\text{\emph{\textbf{x}}}}
\newcommand{\bY}{\text{\emph{\textbf{Y}}}}
\newcommand{\bV}{\text{\emph{\textbf{V}}}}

\newcommand{\bz}{\text{\emph{\textbf{z}}}}
\newcommand{\bu}{\text{\emph{\textbf{u}}}}
\newcommand{\br}{\text{\emph{\textbf{r}}}}
\newcommand{\D}{\mathcal{D}}
\newcommand{\lb}{\llbracket}
\newcommand{\rb}{\rrbracket}
\newcommand{\lrb}{\lb1,P\rb}
\newcommand{\lrbN}{\lb1,N\rb}
 
\newcommand\dt{\frac{d}{dt}}
\newcommand\dtau{\frac{d}{d\tau}}
\newcommand\R{\mathbb{R}}
\newcommand\eps{\varepsilon}
\newtheorem{hyp}{Assumption}
\newtheorem{lemma}{Lemma}
\newtheorem{theo}{Theorem}
\newtheorem{prop}{Proposition}

\begin{document}
\begin{frontmatter}

\title{Derivation of a spatial replicator system with environmental heterogeneity from a co-colonization SIS model with $N$ strains and $P$ patches}

\author[1,*]{Sten Madec}
\author[2]{Erida Gjini}

\address[2]{Center for Computational and Stochastic Mathematics, Instituto Superior Tecnico, University of Lisbon, Lisbon, Portugal}
\address[1]{Institut Denis Poisson, University of Tours, France}
\address[*]{Corresponding author: sten.madec@univ-tours.fr}
	
	\begin{abstract}
		The interplay between local and regional processes in the dynamics of ecological communities remains a challenge to model, analyze and predict. This is especially notable in infectious diseases with multiple strains, where several layers of heterogeneity can interact, including strain biological traits and environmental heterogeneity among locations where disease can spread.
        Motivated by this challenge, here we study a \textit{Susceptible-Infected-Susceptible} (SIS) model with co-colonization and multiple interacting strains where hosts move between a set of inter-connected patches. Under strain similarity and slow migration rate, we derive a fast-slow approximation of the global metacommunity dynamics, resulting in a spatial replicator system for $N$ strains across $P$ patches. 
        In contrast to a discretization approach on the spatial slow-fast PDE originally derived in\citep{le2023spatiotemporal}, here the slow-fast reduction is managed \textit{ab-initio} by a new approach using strongly the  Perron-Frobenius Theorem for Metzler matrices, which simplifies and clarifies the structure of the co-colonization system.
	\end{abstract}
	\begin{keyword}
	 metacommunity \sep coinfection \sep slow-fast dynamics \sep multi-strain \sep SIS model \sep ecological interactions \sep environmental heterogeneity
		
	\end{keyword}
	
	%
	%
	
\end{frontmatter}
	\section{Introduction}
		Here, we consider a multi-patch multi-strain SIS system with co-colonization/coinfection, strain interactions, and host migration between patches. This is an extension from the basic SIS (\textit{Suscetible-Infected-Susceptible}) coinfection model for $N$ interacting strains proposed by \cite{madec2020predicting}, generalized by \citep{le2023quasi} and extended to continuous space in \citep{le2023spatiotemporal}. 
	
	In the non-spatial model versions \citep{madec2020predicting,le2023quasi,le2022disentangling}, it was shown that under strain similarity, strain frequencies follow a slow dynamics given by the replicator equation:
	\begin{equation}\label{le:slowreact}
		\frac{d}{d\tau} z_i = {\Theta z_i \cdot \bigg( \sum_{j\neq i} \lambda_i^j z_j -\mathop{\sum}_{1\leq k<j\leq N} (\lambda_j^k+\lambda_k^j) z_jz_k \bigg)},\quad i=1,\cdots,N
	\end{equation}
	where $\lambda_i^j$ denote pairwise invasion fitnesses between any two strains, and $\Theta$ gives the speed of the dynamics. In the spatial model extension \citep{le2023spatiotemporal}, allowing for diffusion of hosts in continuous space, a similar model reduction was obtained, and in the case of low-diffusion a replicator-like equation was again derived for strain frequencies over space $z_i(x,t)$: 
	
	
	\begin{equation}
		\frac{\partial z_i}{\partial \tau} = \Theta(x) z_i\bigg[(\Lambda(x)\textbf{z})_i - \textbf{z}^T\Lambda(x)\textbf{z}\bigg] + \Vec{\nu}(x) \cdot \nabla z_i + \Delta z_i
		\label{eq:replicatorspace}
	\end{equation}
	where the $z_i$ represent the frequencies of strain $i$ in each point in space, under the effect of local replicator dynamics, diffusion and advection arising from spatial heterogeneity in parameters. 
	
	In the present work, we consider the same epidemiological multi-strain model as in \citep{le2023quasi}, but over discrete space, where we assume hosts can move between a set of fully-connected heterogeneous patches. Leveraging a new analytical method, we show that under the assumptions of $\eps$-quasi neutrality and $\eps$-slow migration for the same $\eps\ll1$, the system is accurately described by a slow dynamics, corresponding to a discrete version of \eqref{le:slowreact}.
	
	This paper is organized as follows. Section 2 introduces the notations and presents the spatial co-infection SIS model with multiple strains, along with some basic properties. Section 3 provides a full analysis of the neutral dynamics in the absence of migration, summarized in Theorem~\ref{thNeutral}.
In Section 4, using these results, we apply Tikhonov’s slow-fast reduction method to derive the equation governing the slow evolution of strain frequencies. This section concludes with the main result of the paper, stated in Theorem~\ref{th:main} and the link with the reaction-advection-diffusion system \eqref{le:slowreact}.
Finally, the paper ends with three appendices, where technical definitions and supplementary details are provided.	\section{Model and assumptions}
	\subsection{Discrete space and notations}
	The space is modeled by $P$ patches. We denote $\lb1,P\rb=[1,P]\cap \mathbb{N}$.
	
	The column vectors $\bz=(\b\bz_p)_{p\in\lrb}$ of $\R^P$ are denoted in bold.
	In particular, $\un=(1,\cdots,1)^T$ and $\zero=(0,\cdots,0)$.
	
	For any vector  $\bX=(X_p)_{p\in \lrb}$ we denote $\bX>0$ if $X_p>0$ for each $p$. For 
	two vectors $\bX,\bY\in\R^P$, we say that $\bX>\bY$ if and only if $\bX-\bY>0$. We use similar notation for  $<,\leq$ and $\geq$.
	We note also the Hadamard product $\bX\bY=(X_pY_p)_{p\in\lrb}$.

	The connection between the patches is described by the $P\times P$ connectivity matrix $\D=(d_{kp})_{1\leq k,p \leq P}$.

	We assume the following.
	
	\begin{hyp}\label{MatrixD}The connectivity matrix $\D\in\R^{P\times P}$ satisfies the three following properties.
	\begin{itemize} 
		\item[(i)] $\D$ is a Metzler matrix. That is $d_{kp}\geq0$ for $k\neq p$. 
		\item[(ii)] 
	$\D$ is irreducible.
\item[(iii)]
	$\D\un=\zero.$ which reads for any $p\in\lrb$:
	$$d_{pp}=-\sum_{\substack{k=1 \\ k\neq p}}^P d_{pk}.$$
	
\end{itemize}
	\end{hyp}

	The point (i) is natural for a connectivity matrix. The point (ii)  is classical and insures that there is always a path from each patch to each other. The point (iii) implies that $\D$ is adapted to a conservation of the density in each patch\footnote{See Appendix C for the case of a matrix adapted to abundance in the case of the connection between tanks of different volume.}.
	Mathematically, these assumptions insure that $0$ is the principal eigenvalue of $\D$ and that all the other eigenvalues have a negative real part (\cite{Bullo2024}).
    
	\subsection{SIS with coinfection (SIDS) in space}
	We are interested in the following SIDS model with $N$ strains co-circulating and $P$ patches between which hosts can move. In any patch $p\in\lb1,p\rb $ and for strains $i,j$ we denote respectively $S_p$, $I_p^i$ and $D_p^{ij}$ the proportion of susceptible,  single infected by $i$ and co-infected by $i$ then $j$ in the patch $p$.
	 We note also
	$$J_p^i=I_p^i +\sum_{j} \mathbb{P}_p^{(i,j)\to (i)} D_p^{ij}+\mathbb{P}_p^{(j,i)\to (i)} D_p^{ji}$$
	where $\mathbb{P}_p^{(i,j)\to (s)}$ is the probability that a host infected by $i$ then $j$ transmit the strain $s$.  We assume here that there is no mutation: $$\mathbb{P}_p^{(i,j)\to (i)}+\mathbb{P}_p^{(i,j)\to(j)}=1.$$
	
	We denote also $\bS=(S_p)_{\in \lrb}$  the vector in $\R^p$.  Hence, by the assumption  \ref{MatrixD} - (iii), $(\D \bS)_p=\sum_{k=1}^P d_{pk} S_k=\sum_{k\neq p} d_{p_k}(S_k-S_p)$. The same notation holds for $\bI^i=(I^i_p)_{p\in\lrb}$ and $\bD^{ij}=(D^{ij}_p)_{p\in\lrb}$.
	The system reads:
	\begin{equation}\label{mainsys}
	\begin{cases}
		\dt S_p=r_p(1-S_p) +\sum\limits_{i=1}^N \gamma_p^i I_p^i +\sum\limits_{1\leq i,j\leq N} \gamma_p^{ij} D_p^{ij}-\sum\limits_{i=1}^N \beta_p^i J_p^i S_p+\diffrate (\D \bS)_p\\
		\dt I_p^i=\beta_p^i J_p^i S_p-(r_p+\gamma_p^i)I_p^i-\sum\limits_{j=1}^N k_p^{ij}\beta_p^i I_p^i J_p^j +\diffrate (\D\bI^i)_p,\; \forall i\in\lb1,N\rb\\
		\dt D_p^{ij}=k_p^{ij}\beta_p^i I_p^i J_p^j-(r_p+\gamma_p^{ij}) D_p^{ij} +\diffrate (\D\bD^{ij})_p,\; \forall (i,j)\in\lb1,N\rb^2
	\end{cases}
	\end{equation}
	 The parameter $\diffrate>0$ is the mean speed of migration and has the unit of  $t^{-1}$. In particular the coefficients $d_{kp}$ of the matrix $\D$ are dimensionless.
	See the table \ref{tabledef} for the definition of all the other parameters. 
	
	Let us state some simple properties of \eqref{mainsys}.
	\begin{prop}\label{prop:inv}[Invariant set]For each patch $p\in\lrb$, denote the total sum: 
		$$\Sigma_p=S_p+\sum\limits_{i=1}^N I^i_p+\sum\limits_{i=1}^N\sum\limits_{j=1}^N D_p^{ij}$$ and $\mathbf{\Sigma}=(\Sigma_p)_{p\in\lrb}$.
		Assume that assumption \ref{MatrixD}-(i) and (iii) holds true. Then the set $$\Omega=\left\{(\bS,\bI^1,\cdots, \bI^N,\bD^{11},\cdots,\bD^{1N},\bD^{21}\cdots,\bD^{NN})\in [0,1]^{P\times (1+N+N^2)},\; \text{s.t.} \mathbf{\Sigma}=\un\right\}$$
		is positively invariant under \eqref{mainsys}.
		
		If moreover the irreducibility assumption \ref{MatrixD}-(ii) holds true  and $\diffrate>0$, then any component   $\bx(t)$
        of the solution $\bX(t)\in\Omega$, satisfies  either $\bx(t)=0$ or $\bx(t)>0$ for all $t>0$. In particular, the set $\Omega_0=\{(\bS,(\bI^i)_i,(\bD^{ij})_{ij})\in\Omega,\bS<1\}$ is positively invariant.
        
         Lastly, $E_0=\{(\un,\zero,\zero)\}$ is invariant and if $\bX(0)\in\Omega\setminus E_0$ then the solution belongs to $\Omega_0$ for any $t>0$.
	\end{prop}
	\begin{proof}
		The non-negativity is classical and result from the structure of the system and the fact that $\mathcal{D}$ is Metzler.
		The positivity under the irreducibility assumption is also classical.

		 The fact that the total biomass $\mathbf{\Sigma}$  is equal to $\un$ is a consequence of  assumption \ref{MatrixD}-(iii) together with the structure of the model. Indeed, summing all the equation of \eqref{mainsys} yields to the vectorial equation
		\begin{equation}\label{total}\dt \mathbf{\Sigma}=\text{diag}(r_1,\cdots,r_P)(\un-\mathbf{\Sigma})+\diffrate\D\mathbf{\Sigma}.
		\end{equation}
		By \ref{MatrixD}-(iii), $\un$ is a stationary solution of \eqref{total} which end the proof. 
	\end{proof}

	\subsection{Main assumptions}
	
	The overall description of the dynamics of \eqref{mainsys} is out of the scope of this paper. We reduce the study to  the quasi-neutral case. 
	Let $\eps>0$ be a small parameters.  
	
	\begin{hyp}[$\eps$-Quasi-Neutrality]\label{qneutral} We assume that all strain-dependent parameters are $\varepsilon$-close (see table \ref{tabledef}) for a small enough $0<\eps\ll 1$.
	\end{hyp}
	
	Even in this case, the impact of $\delta$ change drastically the phenomenon. For $\diffrate\to +\infty$ we obtain an homogeneous SIDS system averaging in space. The method is then standard and similar to the ones described in details in \cite{le2023spatiotemporal}. 
	
	In this paper we assume that the migration is very slow that is, with the same $\eps$ than in assumption \ref{qneutral}.
	\begin{hyp}[$\eps$-Slow diffusion]\label{deps} There exist $d>0$ such that $\diffrate = \varepsilon d$. 
	\end{hyp}	
		
Under these two assumptions, we first analyze the behavior of the system for $\varepsilon = 0$. Then, we investigate the singular limit as $\varepsilon \to 0$, describing how the solution evolves and connecting it to the limiting dynamics.

\section{Strain neutrality and no migration: $\eps=0$.}
In this section we assume that assumptions \ref{qneutral} and \ref{deps} holds true for $\eps=0$. That is the system is fully neutrall (all strains are equivalent) and there is no migration.
\subsection{One strain model without  migration}
If $\eps=0$ and if there is only one strain, then the system is particularly simple since it consists on $P$ SIDS independent systems:

\begin{equation}\label{agregneut}
	\begin{cases}
		\dt S_p=r_p(1-S_p)+\gamma_p I_p+\gamma_p D_p -\beta J_p S_p \\
		\dt I_p= \beta_p J_p S_p -(r_p+\gamma_p)I_p-k_p \beta_p I_p J_p\\
		\dt D_p=k_p\beta_p I_p J_p-(r_p+\gamma_p) D_p 
	\end{cases}
\end{equation}
where $J_p=I_p+D_p$.

The sets  $\Upsilon=\{(S,I,D)\in [0,1]^3,\; S+I+D=1\}$, $e_0=\{(1,0,0)\}$ and $\Upsilon_0=\Upsilon\setminus e_0$ are  positively invariant.

Let us remark that if  $(S_p,I_p,D_p)\in\Upsilon$, it comes $J_p=I_p+D_p=1-S_p$. Hence the dynamics of the first equation of \eqref{agregneut} consist on the single equation 
$$\dt S_p=(1-S_p)(r_p+\gamma_p-\beta_p S_p)$$ \\
The dynamics on each patch $p$ is then straightforward :
\begin{lemma}\label{lemmaneutre}
	Let $p\in\lrb$. And let $(S_p,I_p,D_p)\in\Upsilon$ be a solution of \eqref{agregneut}.
	
\begin{itemize}\item[(i)]	If $\beta_p\leq r_p+\gamma_p$ then $\lim\limits_{t\to+\infty}(S_p(t),I_p(t),D_p(t))=(1,0,0)$. 
	
\item[(ii)]	If $\beta_p>r_p+\gamma_p$ then if $(S_p(0),I_p(0),D_p(0))\in\Upsilon_0$ then $\lim\limits_{t\to+\infty}(S_p(t),I_p(t),D_p(t))= (S_p^{*},I_p^{*},D_p^{*})$
	where
	$$S_p^{*}=\dfrac{r_p+\gamma_p}{\beta_p},\; I_p^{*}=\dfrac{\beta_p (1-S_p^{*})S_p^{*}}{r_p+\gamma_p+k_p\beta_p (1-S_p^{*})},\; D_p^{*}=\dfrac{k_p\beta_p(1-S_p^{*})I_p^{*}}{r_p+\gamma_p}.$$ 
	For latter use, we define also $T_p^*=1-S_p^*=I_p^*+D_p^*$.
	
Moreover $(S_p^{*},I_p^{*},D_p^{*})$ is exponentially stable that is, there exist $M_p>0$ and $\mu_p>0$ such that
$$\forall t\geq 0,\; |S_p(t)-S_p^*|+|I_p(t)-I_p^*|+|D_p(t)-D_p^{*}|\leq M_pe^{-\mu_p t}.$$
\end{itemize}
%
	
\end{lemma}

\subsection{$N-$strain neutral model and $\eps=0$}\label{sec:3.2}
The aim of this section is to study the system \eqref{mainsys} in $\Omega_0$ assuming that $N$ strains are equivalent in their traits, and all parameters are given by the table \ref{tabledef} with $\eps=0$. 

The system reads 
\begin{equation}\label{sysNeutral}
	\begin{cases}
		\dt S_p=r_p(1-S_p)+\gamma_p \sum\limits_{i=1}^N I^i_p+\gamma_p \sum\limits_{(i,j)\in\lb1,N\rb^2}D_p^{ij} -\beta  S_p\sum\limits_{i=1}^N J_p^i \\
		\dt I_p^i= \beta_p J^i_p S_p -(r_p+\gamma_p)I^i_p-k_p \beta_p I_p \sum\limits_{j=1}^NJ_p^j\\
		\dt D_p^{ij}=k_p\beta_p I_p^i J_p^j-(r_p+\gamma_p) D_p^{ij} 
	\end{cases}
\end{equation}

We insist on the fact that in this section the patches are disconnected.

The following result describe the dynamics for the Neutral model.
\begin{theo}[Dynamics under strain neutrality and no migration]\label{thNeutral}
	Assume that $\forall p\in\lrb$, $\beta_p>\gamma_p+r_p$. 
	Let $(\bS,(\bI^i)_{i\in\lrbN},(\bD^{ij})_{(i,j)\in\lrbN^2})$ be a solution of \eqref{sysNeutral} with initial values in $\Omega_0$ and let $(S_p^*,I_p^*,D_p^*)$ be as in the lemma \ref{lemmaneutre}.\\
	Denote the simplex  $\Sigma=\{z=(z^1,\cdots,z^N)\in[0,1]^N,\; \sum_{i=1}^N z^i =1\}$.
	
	Then for any $p\in\lrb$, $\lim\limits_{t\to+\infty}S_p(t)=S_p^*$ and 
	 there exists  a probability distribution $\bz_p=(z_p^1,\cdots,z_p^N)\in\Sigma$, such that 
	
	$$\lim_{t\to+\infty}I_{p}^i(t)= I_p^*z_{i}^p,\quad \lim_{t\to+\infty}D_p^{ij}(t)=D_p^* z_{p}^i z_{p}^j.$$ 
\end{theo}
\begin{proof}
The proof is divided in three steps.

\paragraph{First, aggregated variables} Let $p\in\lrb$ and assume that $\beta_p>r_p+\gamma_p$. Denote $I_p=\sum\limits_{i\in\lb1,N\rb} I_p^i$ and $D_p=\sum\limits_{(i,j)\in\lb1,N\rb^2} D_p^{i,j}$.
Remark that   $\sum\limits_{i\in\lb1,N\rb} J_p^i = I_p+D_p=(1-S_p)$.

Then $(S_p,I_p,D_p)$ satisfies \eqref{agregneut}.
Let us define $X_p(t)=(I_p,D_p)$ and $X_p^*=(I_p^*,D_p^*)^T$ (as defined  in lemma \ref{lemmaneutre}). Since the variables belong to $\Omega_0$ we have $(S_p,I_p,D_p)\in\Upsilon_0$ for any $p\in\lrb$ and the lemma \ref{lemmaneutre} yields 

\begin{equation}\label{limitagreg}
	\lim_{t\to+\infty}X_p(t)= X_p^*.
\end{equation}

\paragraph{Second: the whole attractor}  Define $D_p^i=\frac12\sum\limits_{j=1}^N (D_p^{ij}+D_p^{ji})$ which yields $J_p^i=I_p^i+D_p^i$.

We denote  $X_p^i(t)=(I_p^i(t),D_p^i(t))^T$ and  $T_p(t)=I_p(t)+D_p(t)=1-S_p(t)$.

A short computation shows that $X_p^{i}$ satisfies the non autonomous linear system

\begin{equation}\label{defAP}\dt X_p^i =\begin{pmatrix} 
	\beta_pS_p(t)-(r_p+\gamma_p) - k_p \beta_p T_p(t)& \beta_p S_p(t) \\
	\dfrac12 k_p \beta_p (T_p(t)+I_p(t))& \dfrac12 k_p \beta_p I_p(t) -(r_p+\gamma_p)
\end{pmatrix} X_p^i
\end{equation}

which reads shortly

\begin{equation*}\label{Neutralsys}
	\dt X_p^i = A_p(X_p(t)) X_p^i	
\end{equation*}


From \eqref{limitagreg} we see that $\lim\limits_{t\to+\infty}A_p(X_p(t))= A_p^*$ where

\begin{equation}\label{def-Astar}
	A_p^*=\begin{pmatrix} 
		- k_p \beta_p T_p^*& \beta_p S_p^* \\
		\dfrac12 k_p \beta_p (T_p^*+I_p^*)& \dfrac12 k_p \beta_p I_p^* -(r_p+\gamma_p)
	\end{pmatrix}
\end{equation}
wherein we have set $T_p^*=1-S_p^*=I_p^*+D_p^*$.

This limit $A_p^*$ is clearly an irreducible  Metzler matrix and it is straightforward to verify that
$A_p^* X_p^*= 0$. $0$ is then an eigenvalue of $A_p^*$ with a positive eigenvector $X_p^*$. Then, by the Perron-Frobenius theorem applying to Metzler matrices (for instance Theorem 9.4 in \cite{Bullo2024}), $0$ is the principal eigenvalue of $A_p^*$ and there exists a unique  left eigenvector $\omega_p^*=(\phi_p^*,\psi_p^*)>0$ satisfying $\omega_p^* X_p^*=1$ and $\omega_p^* A_p^*=0$.
Moreover the other eigenvalue\footnote{Here, since $A_p^*$ is a $2\times 2$ matrix, there is only one  other eigenvalue which is explicitly $trace(A_p^*)=-\dfrac12 k_p \beta_p (2T_p^*-I_p^*) -(r_p+\gamma_p)$ which is clearly negative. In a more general setting we cannot know explicitly the other eigenvalue but the conclusion remains the same.} is negative.

An explicit computation gives the following expression

\begin{equation}\label{explicit-omega}
	\omega_p^*=(\phi_p^*,\psi_p^*) \text{ with } \phi_p^*=\dfrac{T_p^*+I_p^*}{2(T_p^*)^2-I_p^* D_p^*} \text{ and } \psi_p^*=\dfrac{2T_p^*}{2(T_p^*)^2-I_p^* D_p^*}.
\end{equation}

Define now $u_p^i(t)=\omega_p^* X_p^i(t)$  the component of $X_p^i$ on the kernel of $A_p^*$.
Note also $\xi_p^i(t)=\Pi_p X_p^i(t):=X_p^i(t)-(\omega^*\cdot X_p^i) X_p^*$ the component of $X_p^i$ orthogonal to $\omega_p^*$.  Remark that for all $t\geq 0$ we have $u_p^i(t)\in[0,\phi_p^*+\psi_p^*]$ and we may write 
$$X_p^i(t)=u_p^i(t) X_p^*+\xi_p^i(t).$$


 we got
\begin{equation}\label{fastdyn}
		\begin{cases}
        \dt \xi_p^i(t)=A_p^* \xi_p^i (t) + \Pi_p (A_p(X_p(t))-A_p^*)X_i^p(t)\\
        \dt u_{p}^i(t)=\omega_p^*\left(A_p(X_p(t))-A_p^*\right)X_p^i(t)
\end{cases}
\end{equation}

By the lemma \ref{lemma:appendixA},   there exists $z_p^i\geq 0$ such that
 $\lim\limits_{t\to+\infty} u_{p}^i(t)= z_{p}^i$. 
 
 Moreover, we have 
 $$\sum_{i=1}^N z_{p}^i = \omega_p^* \lim\limits_{t\to+\infty} \sum_{i=1}^N X_p^i(t) = \omega_p^* X_p^* =1.$$
 
 On an other hand, the spectral properties of $A_p^*$ implies that there exists $\alpha_1>0$ such that $$\forall v\in\Pi_p \R^2,\; \|e^{t A_p ^*} v \| \leq  e^{-\alpha_1 t}\|v\|.$$
 Hence, by virtue of the lemmas \ref{lemma:appendixA} and \ref{lemma:proj}, there exist $M>0$ and $\alpha>0$ such that
\begin{equation}\label{estimate:Xpi}\|X_p^i(t)-z_ p^i X_p^*\|\leq M e^{-\alpha t}.
\end{equation}

\paragraph{Third: back to the original variables}

From \eqref{estimate:Xpi}, then  by definition of $X_p^i$, one gets $I_p^i\to z_p^i I_p^*$ and $J_p^i\to z_p^i T_p^*$. Thus,  the equation 
$$\dt D_p^{ij} = k_p \beta_p I_p^i J_p^j - (r_p+\gamma_p)D_p^{ij}$$ yields
$$\lim_{t \to +\infty} D_p^{ij}= z_p^i z_p^j \dfrac{ k_p \beta_p I_p^* T_p^*}{r_p+\gamma_p}=z_p^i z_p^j D_p^*.$$

\end{proof}
\section{Quasi-Neutrality among strains}
Let $(\bS,(\bI^i)_i,(\bD^{ij})_{ij})\in\Omega_0$ be a solution of \eqref{mainsys}.
We assume also that for any $p\in\lrb$, $\beta_p>\gamma_p+r$.
We follow the two same steps than in the neutral model.

We assume that assumptions \ref{qneutral} and \ref{deps} holds true for a small enough $0<\eps\ll 1$. All the parameters are given in the table \ref{tabledef}.

\begin{table}\label{parameters}
	\centering
	\begin{tabular}{|c|c|c|}
		\hline
		Notation&Quasi-Neutral Formulation&Interpretation \\
		\hline
		\hline
		\multirow{2}{*}{$S_p(t)$}&\multirow{2}{*}{$S_p^*$}&Proportion of susceptible hosts \\
		&&in the patch $p$ \\
		\hline
			\multirow{2}{*}{$T_p(t)=1-S_p(t)$}&\multirow{2}{*}{$T_p^*=1-S_p^*$}&Proportion of all infected hosts\\
		&&in the patch $p$ \\
		\hline
		
		\multirow{2}{*}{$I_p^i(t)$}&\multirow{2}{*}{$I_p^* z_p^i(\eps t)$}&Proportion of hosts in patch $p$\\
		&& singly-infected by strain $i\in [1,..,N]$\\
		
		\hline 
		\multirow{2}{*}{$D_p^{i,j}(t)$}&\multirow{2}{*}{$D_p^* z_p^i(\eps t)z_p^j(\eps t)$}& Proportion of hosts in patch $p$\\&& co-infected first by strain $i$ then $j$.\\
		\hline
		\hline
			\multirow{2}{*}{$\diffrate$}&\multirow{2}{*}{$\eps d$}& Overall host migration rate \\
		&& between patches\\ 
			\hline
		
		\multirow{2}{*}{$r_p$}&\multirow{2}{*}{$r_p$}& Birth rate (equal to death rate) \\
		&&of hosts in patch $p$\\
		\hline
		\hline
		\multirow{2}{*}{$\beta_p^i$}&\multirow{2}{*}{$\beta_p+\eps b_p^i$}& Per-capita host transmission rate \\
		&& of infection by strain $i$ in patch $p$\\
		\hline
		\multirow{2}{*}{$\gamma_p^i$}&\multirow{2}{*}{$\gamma_p+\eps c_p^i$}& Clearance rate of infection in patch $p$\\
		&& for single-infection by the strain $i$ \\
		\hline
		\multirow{2}{*}{$\gamma_p^{i,j}$}&\multirow{2}{*}{$\gamma_p+\eps c_p^{i,j} $}& Clearance rate in patch $p$\\
		&& for co-infection by strains $i$ and $j$\\
		\hline
		\multirow{2}{*}{$k_p^{i,j}$}&\multirow{2}{*}{$k_p+\eps \alpha_p^{i,j}$}& Susceptibility to coinfection by strain $j$ in patch $p$\\
		&& for hosts singly-infected by strain $i$\\
		\hline
		
		\multirow{2}{*}{$\mathbb{P}_p^{(i,j)\to i}=1-\mathbb{P}_p^{(i,j)\to j}$} &\multirow{2}{*}{$\frac12+\eps w_p^{i,j}$}& Probability for a host in patch $p$ \\&&co-infected by the strains $i$ then $j$ to transmit strain $i$\\
		\hline
		
		\hline
	\end{tabular}
	\caption{\textbf{Definition of the key variables and the parameters in the SIS model with coinfection and many strains\eqref{mainsys}.} The second column gives the formulation for the parameters and the state variables in the Quasi-Neutral regime. The formula for $S_p^*$, $I_p^*$ and $D_p^*$ is given in lemma \ref{lemmaneutre}. We use the term infection to refer to propagation of the infectious agent, but the same mathematical description applies for a general colonization process. 
		. }
	\label{tabledef}
\end{table}


\subsection{The aggregated variables}

Denote $I_p=\sum_{i=1}^N I_p^i$ and $D_p=\sum_{(i,j)\in\lrbN^2} D_p^{i,j}$ and $X_p=(I_p,D_p)$. Remark that $\sum_i J_p^i = I_p+D_p=T_p=(1-S_p)$, we get the system on the aggregated variables that we write in the slow time scale $\tau=\eps t$ in order to apply the Tikhonov theorem in the next section \cite{Tikhonov1952}.
\begin{equation}\label{QN-aggreg}
	\begin{cases}
		\eps\dtau S_p=	r_p(1-S_p)+\gamma_p I_p+\gamma_p D_p -\beta J_p S_p +\eps \mathtt{f}_p^0(\tau)+\eps d (\mathcal{D}\bS)_p+o(\eps)\\
		\eps\dtau I_p=\beta_p J_p S_p -(r_p+\gamma_p)I_p-k_p \beta_p I_p J_p+\eps \mathtt{f}_p^1(\tau)+\eps d (\mathcal{D}\bI)_p+o(\eps)\\
		\eps\dtau  D_p=k_p\beta_p I_p J_p-(r_p+\gamma_p) D_p+\eps \mathtt{f}_p^2(\tau)+\eps d (\mathcal{D} \bD)_p+o(\eps)
	\end{cases}	
\end{equation}
where the functions $\mathtt{f}_p^{s}$, $s=0,1,2$ are differential and bounded functions arising from the first-order expansion in $\varepsilon$
 due to the quasi-neutrality assumption \ref{qneutral}.
We recall (see proposition \ref{prop:inv}) that the disease free equilibrium $E_0=(\un,0,0)$ is always a steady state of \eqref{QN-aggreg}.

By perturbation we obtain the following result if the migration is small enough.
\begin{lemma}\label{lemmaqneutre}  Assume that $\beta_p>r_p+\gamma_p$ for each $p\in\lrb$
	There exists $\eps_0>0$ such that for each $\eps\in(0,\eps_0)$,  
	if assumption  \ref{qneutral} and \ref{deps} hold then $E_0$ is unstable and for any initial values in $\Omega\setminus\{E_0\}$, 
	the solution  $(\bS^\eps(\tau),\bI^\eps(\tau),\bD^\eps(\tau))\in\Omega$ of \eqref{QN-aggreg} satisfies
	$\forall \tau>0$, $(\bS^\eps(\tau),\bI^\eps(\tau),\bD^\eps(\tau))=(\bS^*,\bI^*,\bD^*)+\mathcal{O}(\eps)+\mathcal{O}(\eps d).$ More precisely, there exists two continuous functions $M_1,M_2\in C^0\left((0,+\infty)\times [0,\eps_0], (\R^P)^3\right)$ such that
	 $$\forall\tau>0 \text{ and } \eps\in[0,\eps_0],\;	(\bS^\eps(\tau),\bI^\eps(\tau),\bD^\eps(\tau))=(\bS^*,\bI^*,\bD^*)+\eps M_1(\tau,\eps)+\eps d M_2(\tau,\eps).$$
	
\end{lemma}

\subsubsection{The reduction}

By virtue of the Tikhonov theorem, we obtain the main result of this paper.
\begin{theo}\label{th:main}
Denote $(\bS,\bI^{1},\cdots, \bI^N,\bD^{11},\cdots,\bD^{1N},\bD^{21}\cdots,\bD^{NN})\in \Omega_0$  the  solution of 
	 \eqref{mainsys}. 
Denote also the simplex $\Sigma=\{\bz\in[0,1]^N,\;\sum_{i=1}^N z^i =1\}.$

	 Assume that the assumption \ref{MatrixD} on $\mathcal{D}$ is satisfied.
	 Then, there exists $\eps_0>0$ such that for each $\eps\in (0,\eps_0)$,  if the assumptions \ref{qneutral} and \ref{deps} hold true for the system \eqref{mainsys}, 
	 there exists $\tau_0>0$ and $\bz_0$ such that for any $T>\tau_0$, any $\tau\in[\tau_0,T]$
	  and   $\forall (i,j)\in\lrbN^2$:
	 $$\bS\left(\dfrac{\tau}{\eps}\right)=\bS^*+O(\eps),\quad \bI^i\left(\dfrac{\tau}{\eps}\right)=\bz^i(\tau)\bI^*+O(\eps),\quad \bD^{ij}\left(\dfrac{\tau}{\eps}\right)=\bz^i(\tau)\bz^j(\tau)\bD^*+O(\eps)$$

	where $\bz=(\bz_p)_{p\in\lrb}\in\Sigma^P=\Sigma\times\cdots\times\Sigma$ is the solution of \begin{equation}\label{finalequation-mat}\begin{cases}
	\dtau z_{p}^i=\Theta_p z_{p}^i\left((\Lambda_p \bz_p)_i-\bz_p \Lambda_p \bz_p\right)+ d (\mathcal{M}\bz^{i})_p,\\
    \bz(0)=\bz_0\in\Sigma^P,
    \end{cases}
\end{equation}	

wherein we have set for each $p\in\lrb$,$\Theta_p>0$ and  $\Lambda_p=(\lambda_p^{ij})_{(i,j)\in\lrbN^2}$
as in \eqref{def:lambda} and where
$\mathcal{M}=\left(m_{pk}\right)_{(p,k)\in\lrb^2}$  is a Metzler matrix defined by $m_{pk}=\left(d_{pk} \left(\omega_p^*\cdot X_k^*\right)\right)$ for $p\neq k$ and 
$m_{pp}=-\sum_{k\neq p} m_ {pk}$.

\end{theo}
\begin{proof}
    As in the section \ref{sec:3.2}, for any $p\in\lrb$ and $i\in\lrbN$, define $D_i^p=\frac12\sum_{j=1}^N (D^{ij}_p+D_p^{ji})$ and $X^i_p=(I^i_p,D_p^i)^T$. Denote also $\bX^i=(\bI^i,\bD^i)^T\in [0,1]^{2P}$ and $X_p^\eps(\tau)=\sum\limits_{i=1}^N X_p^i\left(\dfrac{\tau}{\eps}\right)=(I_p^\eps,D_p^\eps)^T$.

\begin{equation}\label{eq:slowepsX} \eps \dtau	 X_p^i =A_p(X_p^\eps(\tau)) X_p^i+\eps \mathcal{F}_p^i(X_p^1,\cdots,X_p^N) +\eps d\left(\begin{pmatrix} \mathcal{D}&0\\0&\mathcal{D}\end{pmatrix} \bX^i\right)_p+o(\eps)\end{equation}
where  $\mathcal{F}_p^i(X_p^1,\cdots,X_p^N) \in[0,1]^2$ come from the first order expansion in $\eps$ due to the quasi-neutrality assumption \ref{qneutral}.
 
From \eqref{defAP} and lemma \ref{lemmaqneutre}, we can write 
$$A_p(X_p(\tau))=\underbrace{A_p^*}_{Neutrality}+\underbrace{\eps U(\tau)}_{\text{perturbation from strains variability}}+\underbrace{\eps dV(\tau)}_{\text{perturbation from hosts migration}}+o(\eps)$$
where $U(\tau)$ results from the perturbations on $X_p$ independant of $d$ and $\eps d V(\tau)$ comes from  the perturbations on $X_p$ due to the migration terms.
 This yields

$$\eps\dtau X_p^i =\left(A_p^*+\eps U(\tau)+\eps dV(\tau)\right) X_p^i+\eps \mathcal{F}_p^i(X_p^1,\cdots,X_p^N) +\eps d\left(\begin{pmatrix} \mathcal{D}&0\\0&\mathcal{D}\end{pmatrix} \bX^i\right)_p+o(\eps)$$

Let  $\omega_p^*=(\phi_p^*,\psi_p^*)$ defined in \eqref{explicit-omega} and $u_{p}^i=\omega_p^* X_p^i$. We have $X_p^i(\tau)=u_{p}^i(\tau) X_p^*+\xi_p^i(\tau)$ with $\omega_p^*\xi_p^i=0$. 
This yields\footnote{Recall that $0$ is the principal eigenvalue of $A_p^*$, the other one being $-\alpha_p<0$ for some $\alpha_p>0$.  By construction $\xi_p^i$ belongs to the eigen-space corresponding to $-\alpha_p$.  It follows that $\dt \xi_p^i=A_p^* \xi_p^i+ O(\eps)$ reads     $\dt \xi_p^i=-\alpha_p \xi_p^i + O(\eps)$ .} the slow-fast system


\begin{equation}\label{eq:slow-fast}\begin{cases}
	\eps\dtau \xi_p^i=-\alpha_p \xi_p^i + O(\eps)\\
\dtau u_{p}^i=\left[\omega^* \left(U(\tau)+dV(\tau)\right)  X_p^*\right]u_{p}^i+ d \left(\phi_p^* \left(\mathcal{D} (\bI^* \bu^i)\right)_p+\psi_p^* \left(\mathcal{D} (\bD^*\bu^i)\right)_p\right) +\mathcal{F}_p^i(u_{p}^1X_p^*+\xi_p^1,\cdots,u_{p}^NX_p^*+\xi_p^N) \\
\phantom{\dtau u_{p}^i=+}
+\omega^* \left(U(\tau)+dV(\tau)\right) \xi_p^i+
d \left(\phi_p^* \left(\mathcal{D} (\bI^* \mathbf{\xi}^i)\right)_p+\psi_p^* \left(\mathcal{D} (\bD^*\mathbf{\xi}^i)\right)_p\right)+O(\eps).
\end{cases}
\end{equation}

Now we apply the Tikhonov slow-fast method \cite{Tikhonov1952}. Remark that comming back to $t=\dfrac{\tau}{\eps}$ and taking $\eps\to$ we obtain the fast dynamics \eqref{fastdyn} described in the proof of the Theorem \ref{thNeutral}.

Hence, let $(\mathbf{\xi}_p^{i,\eps}(\tau),\bu_p^{i,\eps}(\tau))$ be a solution of \eqref{eq:slow-fast}. For $\tau>0$, we got $(\mathbf{\xi}_p^{i,\eps}(\tau),\bu_p^{i,\eps}(\tau))\to (0,z_p^i(\tau))$ as $\eps \to 0$.
Define also
 $\bz^{i}=(z^{i}_p)_{p\in\lrb}\in\R^P$ 
 and $\bz_{p}=(z_p^i)_{i\in\lrbN}\in\R^N$.
 
We have $z_p^i(\tau) $ satisfy for all $\tau >0$ :
$$\sum_{i=1}^N z_p^i (\tau)=\lim_{\eps\to 0}\sum_{i=1}^N u_p^{i,\eps} (\tau)=
\omega_p^*\lim_{\eps\to 0} \sum_{i=1}^N X_p^{i,\eps}(\tau) =
\omega_p^*\lim_{\eps\to 0} X_p^{\eps}(\tau) =\left(\omega^*\cdot X_p^*\right)=1
$$

In other words,  for any $p\in\lrb$, we have $\bz_p\in \Sigma=\{\bz\in[0,1]^N,\; \sum_{i=1}^Nz^i=1\}$.

Moreover, taking $\eps\to 0$ in \eqref{eq:slow-fast} and $\xi_p^i\to0$, we obtain the slow equation  on $(\bz_1,\cdots,\bz_{P})\in\Sigma^P$:
\begin{equation}\label{zslow_start}\dtau z_{p}^i=z_{p}^i\left(\omega_p^*U(\tau) X_p^* +  f_p^i(\bz_p)\right)+ d \left(
\phi_p^*\left( \mathcal{D} (\bI^* \bz^{i})\right)_p+\psi_p^* \left(\mathcal{D} (\bD^*\bz^{i})\right)_p+\nu_p(\tau) z_{p}^i\right)
\end{equation}
 wherein we have set $\nu_p(\tau)=\omega_p^* V(\tau)X_p^*$ and\footnote{A direct computation shows that $z_p^i$ is in factor of this expression.} $z_p^if_p^i(z)=\mathcal{F}_p^i(z_{p}^1X_p^*,\cdots,z_{p}^NX_p^*)$.

The first parenthesis is independent on $d$. Since $\bz_p(\tau)\in\Sigma$ for any $p$ and $\tau>0$ and $d\geq 0$, taking $d=0$ and summing over $i$ yields 
\begin{equation}
0=\omega_p^*U(\tau) X_p^*+\bar{f}_p(\bz_p)
\end{equation}
where we have denote $\bar{f}_p(\bz_p)=\sum\limits_{i=1}^N z_p^i f_p^i(\bz_p)$. It follows that  
$z_{p}^i\left(\omega_p^*U(\tau) X_p^* +  f_p^i(\bz_p)\right)=z_p^i(f_p^i(\bz_p)-\bar{f}_p(\bz_p))$
Which is the term appearing in a replicator equation.

For the explicit computation it suffice to compute explicitly the functions $f_p^i$ which come from the explicit first order expansion $\mathcal{F}_p^i$ in \eqref{eq:slowepsX}. The explicit computation is detailed in  \cite{le2023quasi} wherein   the system without patches is computed. It appears that the function $f_p^i$ are linear and by denoting $\Lambda_p=(\lambda_p^{ij})_{i,j\in\lrbN^2}$ the pairwise fitnesses matrix  defined in the above-cited paper (see the appendix \ref{Appendix:Lambdap} for the full formula), the first parenthesis in the right term of \eqref{zslow_start} reads 

$$z_{p}^i\left(\omega_p^*U(\tau) X_p^* +  f_p^i(\bz_p)\right)=\Theta_p z_p^i\left((\Lambda_p \bz_p)_i-\bz_p \Lambda_p \bz_p\right)$$

For the second parenthesis with $d$ in factor, remarks first that we have explicitly (denoting $\big(\;\cdot\;)$  the Euclidean inner product in $\R^2$)
$$\phi_p^*\left( \mathcal{D} (\bI^* \bz^{i})\right)_p+\psi_p^* \left(\mathcal{D} (\bD^*\bz^{i})\right)_p=
\sum_{k=1}^P d_{pk} (\phi_p^* I_k^*+\psi_p^* D_k^*) z_{k}^i=\sum_{k=1}^P d_{pk} \left(\omega_p^*\cdot X_k^*\right) z_{k}^i.$$

\noindent Finally, since the Cartesian product  $\Sigma^P=\Sigma\times\cdots\times\Sigma$ is invariant under the equation \eqref{zslow_start}, summing over $i$ for each $p$ yields

$$\nu_p(\tau)=-\sum\limits_{k=1}^N  d_{pk} \left(\omega_p^*\cdot X_k^*\right).
$$
This gives the final slow equation defined on $\Sigma^P=\Sigma\times \cdots\times \Sigma$:

\begin{equation}\label{finalequation}
	\begin{cases}\dtau z_p^i=\Theta_p z_p^i\left((\Lambda_p \bz_p)_i-\bz_p\Lambda_p \bz_p\right)+ d \sum_{k=1}^P d_{pk} \left(\omega_p^*\cdot X_k^*\right)(z_k^i-z_p^i)\\
    \forall p\in\lrb,\; \sum\limits_{i=1}^N z_p^i=1
    \end{cases}
\end{equation}	

\noindent It is possible to write this equation in a more compact form. Define the  matrix $\mathcal{M}=(m_{pk})_{(p,k)\in\lrb^2}$ by $m_{pk}=\left(d_{pk} \left(\omega_p^*\cdot X_k^*\right)\right)$ for $p\neq k$ and  
$m_{pp}=-\sum_{k\neq p} m_ {pk}$. For $k\neq p$, we have $d_{pk}\geq 0$ (assumption \ref{MatrixD}-(i)))  and $\omega^*_p$ and $X_k^*$ are both positive (they are the positive principal eigenvector of a Metzler matrix, see the proof of theorem \ref{thNeutral}). Then   $\mathcal{M}$ is Metzler  and inherits the irreducibility of $\mathcal{D}$. We obtain the shorter form of the reduced equation for $\bz\in\Sigma^P=\Sigma\times\cdots\Sigma$:

\begin{equation}\label{finalequation-matbis}
	\dtau z_{p}^i=\Theta_p z_{p}^i\left((\Lambda_p \bz_p)_i-\bz_p \Lambda_p \bz_p\right)+ d (\mathcal{M}\bz^{i})_p.
\end{equation}

\end{proof}
\subsection{Patch heterogeneity and the matrix $\mathcal{M}$. }
We finish this paper with the following remark which links this result with the continuous spatial structure model described in \cite{le2023spatiotemporal} and the companion paper \cite{maroco2025multipatch}.

Note that if $w_p^*$ and $X_p^*$ are both independent of the patch  $p$ then for any $p,q$, we have $(\omega_p^* \cdot X_q^*)=(\omega_p^* \cdot X_p^*)=1$ and therefore $\mathcal{M}=\mathcal{D}$. This is in particular true if $A_p^*$ does not depend on $p$.
Hence the terms $(\omega_p^* \cdot X_q^*)$ in the migration matrix $\mathcal{M}$ reflect the spatial heterogeneity of the aggregated variables. To highlight this fact we may rewrite 

$$\omega_p^* \cdot X_k^*=\omega_p^*\cdot X_p^*+\omega_p^* \cdot (X_k^*-X_p^*)=1+\omega_p^* \cdot (X_k^*-X_p^*),$$
which, denoting $\nu_{pk}=d_{pk} \left(\omega_p^*\cdot(X_k^*-X_p^*)\right),$ yields:
\begin{equation}\label{finalequation-adv}
	\dtau z_p^i=\Theta_p z_{p}^i\left((\Lambda_p \bz_p)_i-\bz_p \Lambda_p \bz_p\right)+ d (\mathcal{D}\bz^{i})_p+d\sum_{k=1}^P d_{pk} \nu_{pk}(z_k^i-z_p^i).
\end{equation}	

This is exactly the system that we obtain in the companion paper \citep{maroco2025multipatch} if we were to reinterpret the continuous-space replicator found in  \cite{le2023spatiotemporal} by replacing the diffusive operator by the matrix $\mathcal{D}$ and the advection term $\vec{\nu}(x)\cdot \nabla z_i (x)$ by the operator  $\sum_k d_{pk} \left(\omega_p^*\cdot(X_k^*-X_p^*)\right)(z_k^i-z_p^i)$. This means that the two approaches: discretization of the PDE slow-fast reduction, or slow-fast method on the discrete multi-patch ODE system, are equivalent, and produce the same final discrete space replicator equation, that we highlight here.

\section{Conclusion}


In this paper, we derive a spatial Replicator equation from a multi-strain SIS model with coinfection and spatial structuration. This derivation relies on two asymptotic assumptions: quasi-neutrality (the strains are similar but not identical) and slow migration.
We propose a new approach by rewriting the system in a form involving a Metzler matrix, whose powerful properties are key to the reduction. This approach is very general and can be extended to various coinfection structures.
The resulting Replicator equation is new in the literature, and its study may help in understanding the impact of space on multi-strain interactions.

{\it Commutativity of the methods.}
A preliminary observation is that the same discrete system emerges when applying a spatial discretization to the spatial replicator model introduced in \cite{le2023spatiotemporal}, which itself is obtained via a slow-fast reduction from a continuous reaction-diffusion SIS framework. This reveals a form of commutativity between the two procedures: whether one first performs the reduction and then discretizes, or discretizes before applying the reduction, the resulting system remains unchanged.

{\it On the replicator equation in space.}
It is worth noting that spatial replicator systems are relatively uncommon in the literature. Two main reasons may account for this. 
First, unlike in the aspatial case, when starting from a Generalized Lotka–Volterra system with migration expressed in terms of abundances $\mathbf{n}^i$, the total population size $\mathbf{N} = \sum \mathbf{n}^i$ typically varies across space. Consequently, rewriting the system in terms of densities $\bz^i = \mathbf{n^i} / \mathbf{N}$ does not straightforwardly yield a replicator equation with migration. This difficulty led early spatial replicator models to combine variables of abundance $\mathbf{n}^i$ and frequency $\bz^i$ (see \cite{Vicker89,Vicker91}). This assumption can be relaxed by introducing an additional equation for the total biomass, as proposed in \cite{DurretLevin94} and generalized in \cite{Griffin21,Griffin23}. 

A second line of reasoning, developed by \cite{Bratus2014}, highlights a conceptual difficulty from a game-theoretic perspective: it is hard to justify the locality of the quadratic term $\bz^T \Lambda \bz$ in the replicator equation, since payoffs are defined relative to all individuals. To address this, they proposed and studied spatial replicator models in which the total payoff is integrated over space, with migration occurring either within or outside the payoff structure. 

In contrast, motivated by SIS models with co-colonization \citep{madec2020predicting,le2023quasi,le2023spatiotemporal}, we directly obtain a replicator system with a local payoff and simple migration acting on the frequencies. Interestingly, the spatial heterogeneity of global quantities imposes a modification of the migration matrix via the coefficients $\nu_{kp}$, but without altering the structure of the equation itself.

{\it Method limitations.}
In this paper, we assume super-critical basic reproduction number in each patch, i.e. $R_{0,p}:=\frac{\beta_p}{r_p+\gamma_p} > 1$, $\forall p$; that is, the disease persists locally everywhere, even in the absence of migration between patches. Clearly, if the disease persists in one patch but not in another, exchanges between patches may lead either to extinction or persistence of the disease. It is therefore natural to ask what the quasi-neutral dynamics are in such cases. Strictly speaking, the current method fails here, as it relies explicitly on the exponential convergence to the endemic equilibrium within each patch in the absence of diffusion. In particular, the very definition of $z_p^i$ breaks down when $R_{0,p} < 1$. Both the statement and the proof of Theorem \ref{th:main} must be significantly adapted in this context.

Similarly, our approach relies heavily on the fact that the epidemiological system is formulated in terms of densities. This requires that the diffusion coefficient remains the same for all the species together with the assumption $\mathcal{D} \un = \zero$  on the connectivity matrix $\mathcal{D}$. As shown in the appendix, this assumption is not general, which implies that it is not always possible to rewrite the system in terms of host densities. A more general framework would be to work directly with a system expressed in terms of abundances (numbers of hosts). It would be interesting to investigate the nature of the slow dynamics in this setting.

{\it Qualitative behavior.} This new spatial replicator equation may help elucidate the effect of discrete space and environmental heterogeneity on the dynamics of multi-strain epidemiological systems, or multispecies co-colonization systems. The final equation consists merely of a coupling of several replicator systems through the connectivity matrix $\mathcal{M}$. While the resulting dynamics are, of course, potentially very complex, the structural simplicity of the replicator equation allows for analytical treatment in certain cases. In particular, when only two species are involved, the system becomes cooperative and is essentially characterized by its steady states. This specific case is investigated in the companion paper \citep{maroco2025multipatch}.

\bibliographystyle{abbrv}
\bibliography{article_refs}

\appendix
\section{Two classical lemmas}
This lemma is a classical result of real analysis.
\begin{lemma}\label{lemma:appendixA}
	Let $A,B\in\R$ such that $A<B$. Suppose that  $x\in C^1(\R,]A,B[)$ satisfies $|x'(t)|\leq M e^{-\alpha t}$ for all $t\in\R $ and some $M,\alpha>0$.\\
	
	Then there exist $z\in[A,B]$ such that $ |x(t)-z|\leq \dfrac{M}{\alpha} e^{-\alpha t}$.
\end{lemma}
\begin{proof}
Let $t_1,t_2\in\R$. By the assumption on the derivative it comes $|x(t_2)-x(t_1)|\leq M\left|\int_{t_1}^{t_2} e^{-\alpha s}ds\right|=\dfrac{M}{\alpha}\left|e^{-\alpha t_1}-e^{-\alpha t_2}\right|$. So for any $\eps>0$ there exist $T>0$ such that for any $t_1,t_2>T$, $|x(t_2)-x(t_1)|\leq \eps$. Hence $x$ is a uniform Cauchy function defined on the compact $[A,B]$ and there exists $z\in [A,B]$ such that $\lim_{t\to+\infty} x(t)\to z$.

Moreover we have
$$|x(t)-z|=\left|\int_t^{+\infty} x'(s)ds\right|\leq \dfrac{M}{\alpha}{e^{-\alpha t}}.$$

\end{proof}

\begin{lemma}\label{lemma:proj}
 Let $A$ be a $n\times n$ matrix. Assume that $E$ is a linear subspace of $\mathbb{R}^n$ invariant under $A$ and that there exist  $\alpha_1 >0$ such that $$\forall t\geq 0,\; \forall X\in E, \; \|e^{t A} X\|\leq  e^{-\alpha_1 t}\|X\|.$$

 Let $G: \R\to E$ be a map satisfying $\|G(t)\|\leq M_2 e^{-\alpha_2 t}$
 for any $t\geq 0$ and some positive numbers $M,\alpha_2$.

 Consider the equation 
 $$\begin{cases}
     \dt \xi = A \xi + G(t),\\
     \xi(0)=\xi_0\in E.
 \end{cases}$$
 Then $\xi(t)\in E$ for all $t\geq 0$ and there exists $M>0$ and $\alpha>0$ such that

 $$\forall t\geq 0,\; \|\xi(t)\|\leq M e^{-\alpha t}. $$
\end{lemma}
\begin{proof}We write 
  $$\xi(t)=e^{t A} \xi_0+\int_0^te^{(t-s) A }  G(t) ds. $$
  So $$\|\xi(t)\|\leq e^{-\alpha_1}\|\xi_0\|+Me^{-\alpha_1 t}\int_0^t e^{(\alpha_1-\alpha_2)s} ds$$
  and the conclusion follows.
\end{proof}

\section{Formula of the Pairwise invasion Fitness matrix between strains in each patch $\Lambda_p$}\label{Appendix:Lambdap}
This formula is taken from \cite{le2023quasi}. We reproduce it here for the sake of completeness.
We use  the notation of lemma \ref{lemmaneutre} table \ref{tabledef}.
 
The average speed of the dynamics in the patch $p$ is given by
$\Theta_p=\sum\limits_{s=1}^5 \Theta_{p,s}$
where    
$$\Theta_{p,1} = \frac{2(r_p+\gamma_p)\big(T_p^*\big)^2}{\mathcal{P}_p}, \quad  
\Theta_{p,2} = \frac{\gamma_p I_p^*(I_p^* + T_p^*)}{\mathcal{P}_p}, \quad
\Theta_{p,3} =\frac{\gamma_p T_p^* D_p^*}{{\mathcal{P}_p}} ,
\Theta_{p,4} = \frac{2(r_p+\gamma_p) T_p^* D_p^*}{{\mathcal{P}_p}}, \quad 
\Theta_{p,5} = \frac{\beta_pI_p^* T_p^*}{{\mathcal{P}_p}}$$
with $${\mathcal{P}_p}=2\big(T_p^*\big)^2-I_p^*D_p^*.$$

And the coefficients of the matrix appearing in the Replicator equation $\Lambda_p=(\lambda_p^{ij})_{(i,j)\in\lrbN^2}$ are given by 
\begin{equation}\label{def:lambda}\lambda_p^{ij} = \theta_{p,1}(b_p^i - b_p^j) + 
	\theta_{p,2}(-\nu_p^i + \nu_p^j) + 
	\theta_{p,3}(-c_p^{ij} - c_p^{ji} + 2c_p^{jj}) + 
	\theta_{p,4}(w_p^{ij} - w_p^{ji}) + 
	\theta_{p,5}(I_p^*(\alpha_p^{ji} - \alpha_p^{ij}) + {D_p^*}(\alpha_p^{ji} - \alpha_p^{jj}))\end{equation}
{where the coefficients setting the `weights' of each trait dimension are: $\theta_{p,s}=\dfrac{\Theta_{p,s}}{\Theta_p}$ for $s=1,\cdots,5$.}

\section{Remark on the connectivity matrix $\mathcal{D}$}
\subsection{Conservation of the total mass}
Initially we assumed that $\mathcal{D}\un=0$. This assumption is consistent with a model in density. If we want to write model in term of abundance the correct assumption is 

\begin{equation}\label{conservation-M}\un^T \mathcal{D}=0.
\end{equation}

In particular, if we model $P$ tanks of volume $V_1(t),\cdots, V_p(t)$ at the time $t$ and set $\bV(t)=(V_1(t),\cdots,V_p(t))^T$ then we may write 

$$\dt \bV(t)=\mathcal{D} \bV(t)$$

Then the assumption \eqref{conservation-M} states simply that the total volume $\un^T \bV(t)$ is constant: their is no addition or removing of volume from the system.
But the volume in each patch may change. In particular, $\bV$ is not necessarily a steady state of the system

\begin{equation}\label{AppendixC-dyn}\dt \bS=\br (\bV-\bS)+\diffrate\mathcal{D}\bS
	\end{equation}

To insure that $V_p$ is constant for each $p$ me must add the additional structure on $\mathcal{D}$:
\begin{equation}\label{conservation-V}
\mathcal{D} \bV=0
\end{equation}

With such a structure,  $\bV$ is then a stable steady state of the system 

\eqref{AppendixC-dyn}

\subsection{Renormalisation}
Now assume that $\mathcal{D}$  satisfies \eqref{conservation-M} and \eqref{conservation-V}. We may renormalised each patch by its volume by the change of variables

$$x_i= y_i V_i$$

Hence we define $P_{\bV}=diag(\bV)$ and
$$\hat{\mathcal{D}} =P_{\bV}^{-1} \mathcal{D} P_{\bV}$$

Then $\hat{\mathcal{D}}^T$ satisfies \eqref{conservation-M} and \eqref{conservation-V}. 

Applying this change of unit on \eqref{AppendixC-dyn}
by denoting $S=P_{\bV}\hat{\bS}$ we get

$$\dt \hat{\bS}=\br (\un-\hat{\bS})+\diffrate\hat{\mathcal{D}}\hat{\bS}$$
This is exactly the type of matrix we are using in order to insure the conservation of the density in each patch.

In conclusion, we may always write this system if the migration keeps invariant the mass of each patch. Otherwise, we need to work directly on  \eqref{AppendixC-dyn}.

\subsection{Construction of $\mathcal{D}$ for fixed volumes}

Indeed, for a given irreducible Metzler Matrix satisfying \eqref{conservation-M}, $0$ is the principal eigenvalue and then there exists only one such $\bV$, up to a multiplicative constant.

Conversely, we can ask what type of matrix keep fixed a given family of volumes $\{V_1,\cdots,V_p\}$. Next, we show that we  may always explicitly construct such a (family of) matrix.
If $\bV$ is given, then we can construct a Metzler matrix satisfying \eqref{conservation-M}. 

It is easy to see that, for $P$ patches, the space of the matrices (not necessary Metzler) solution of the problem is linear and of dimension $P^2-2P+1=(P-1)^2$.

If $P=2$ then there is no choice. Denoting $\bV=(V_1,V_2)$ we get a space of dimension 1 and the solution reads

$$M_{12}=x \begin{pmatrix} -V_2&V_1\\V_2&-V_1\end{pmatrix}, \;x\in\R.$$ 
For any $x>0$ we obtain a Metzler matrix solution of the problem.

In general, denoting $\bV=(V_1,\cdots,V_p)$ and for $i<j$, the matrix $M_{ij}\in\R^{P\times P}$  defined by 

$$M_{ij}(i,i)=-V_j,\quad M_{ij}(i,j)=V_i,\quad M_{ij} (j,i)=V_j,\quad M_{ij} (j,j)=-V_i$$
and $0$ elsewhere.

Then for any non-negative numbers $x_{ij}$ the matrix

$$M=\sum_{1\leq i<j\leq P} x_{ij} M_{ij}  
$$ satisfies \eqref{conservation-M} and \eqref{conservation-V}.
Note that this gives only $P(P-1)/2$ liberty degree on the $(P-1)^2$, so other matrices may be  solution as well.

\end{document}